\documentclass[12pt]{iopart}
\usepackage{amsfonts,amssymb}
\usepackage{hyperref}
\usepackage[margin= 1in]{geometry}
\long\def\nodo#1{}
\def\id{\mathrm{id}}

\def\genfd{\mathbf{k}}

\usepackage{amsthm}
\newtheoremstyle{definition}{}{}{\upshape}{}{\bfseries}{.}{0.5em}{}
\theoremstyle{definition}
\newtheorem{theorem}{Theorem}[section]
\newtheorem{lemma}[theorem]{Lemma}

\newtheorem{definition}[theorem]{Definition}

\begin{document}
\comment{Comment on ``Twisted bialgebroids versus bialgebroids from a Drinfeld twist''}

\author{Zoran \v{S}koda$^1$, Martina Stoji\'{c}$^2$}

\address{$^1$Faculty of Teacher’s Education, University of Zadar, F.~Tudjmana 24,
	23000 Zadar, Croatia}
\address{$^2$Department of Mathematics, Faculty of Science, University of Zagreb,
	Bijeni\v{c}ka cesta~30, 10000 Zagreb, Croatia}
\ead{$^1$zskoda@unizd.hr}
\ead{$^2$stojic@math.hr}
\begin{abstract}
  A class of left bialgebroids whose underlying algebra $A\sharp H$ is a smash product of a bialgebra $H$ with a braided commutative Yetter--Drinfeld $H$-algebra $A$ has recently been studied in relation to models of field theories on noncommutative spaces. In [A.~Borowiec, A.~Pacho\l, ``Twisted bialgebroids versus bialgebroids from a Drinfeld twist'', J.\ Phys.\ A50 (2017) 055205] a proof has been presented that the bialgebroid $A_F\sharp H^F$ where $H^F$ and $A_F$ are the twists of $H$ and $A$ by a Drinfeld 2-cocycle $F = \sum F^1\otimes F^2$ is isomorphic to the twist of the bialgebroid $A\sharp H$ by the bialgebroid 2-cocycle $\sum 1\sharp F^1\otimes 1\sharp F^2$ induced by $F$. They assume $H$ is quasitriangular, which is reasonable for many physical applications. However the proof and the entire paper take for granted that the coaction and the prebraiding are both given by special formulas involving the R-matrix. 
There are counterexamples of Yetter--Drinfeld modules over quasitriangular Hopf algebras which are not of this special form. Nevertheless, the main result essentially survives. We present a proof with a general coaction and the correct prebraiding, and even without the assumption of quasitriangularity.
\end{abstract}

\vspace{2pc}


\section{Introduction}
\label{sec:intro}

Associative bialgebroids and Hopf algebroids appear as algebraic models of noncommutative phase spaces~\cite{borowpachol,lukierskicov,halgoid} and in other roles related to symmetries of noncommutative spaces~\cite{bohmHbk,hanmajid,lu}, inclusions of subfactors~\cite{kadszl} and deformation quantization~\cite{xu}. Drinfeld twists are often a source of new examples of Hopf algebras. Article~\cite{borowpachol} shows how the Xu's variant~\cite{xu} of Drinfeld twists of bialgebroids of the form of a smash product of a bialgebra $H$ and a braided commutative Yetter--Drinfeld $H$-module algebra $A$ may be induced from a Drinfeld twist of $H$, and how so twisted bialgebroid compares to the bialgebroid defined by the smash product of the appropriately twisted $H$ and the twisted $A$. This article is to slightly correct the arguments in~\cite{borowpachol} and complete this comparison. 

For a bialgebra $H$ over a field $\genfd$, a left-right Yetter--Drinfeld (YD) module $M$ is a $\genfd$-vector space with a left $H$-action $h\otimes m\mapsto h\triangleright m$ and a right $H$-coaction $\rho\colon m\mapsto \sum m_{[0]}\otimes m_{[1]}$ satisfying the YD compatibility condition~\cite{radfordtowber}
\begin{equation}\label{eq:YD}
  (h_{(1)}\triangleright m_{[0]})\otimes h_{(2)}m_{[1]} = (h_{(2)}\triangleright m)_{[0]}\otimes (h_{(2)}\triangleright m)_{[1]}h_{(1)}.
\end{equation}
Here and below we often omit the summation sign when using Sweedler notation~\cite{Majid}. 
Morphisms of YD modules are $\genfd$-linear maps which are both morphisms of $H$-modules and of $H$-comodules. The category of left-right Yetter--Drinfeld $H$-modules ${}_{H}\mathcal{YD}^{H}$ has two standard monoidal structures~\cite{radfordtowber}, and article~\cite{borowpachol} chooses one in which the vector space $M\otimes N$ has the $H$-action $h\triangleright (m\otimes n) = \sum (h_{(1)}\triangleright m) \otimes (h_{(2)}\triangleright n)$ and the $H$-coaction
\begin{equation}\label{eq:rhotensor}
  m\otimes n\mapsto \textstyle\sum (m_{[0]}\otimes n_{[0]})\otimes n_{[1]} m_{[1]}.
\end{equation}
This tensor product has a prebraiding with the components
\begin{equation}\label{eq:sigma}
  \sigma_{MN}\colon M\otimes N\to N\otimes M,
  \,\,\,\,\,m\otimes n\mapsto \sum n_{[0]}\otimes (n_{[1]}\triangleright m).
\end{equation}
A monoid $A$ in ${}_{H}\mathcal{YD}^{H}$ is a YD module with a product $\mu\colon a\otimes b\mapsto a\cdot b$ which makes it an $H$-module algebra and an $H^{op}$-comodule algebra; these monoids are called Yetter--Drinfeld $H$-module algebras. A monoid $A$ is braided commutative if $\mu\circ\sigma_{AA}= \mu$, which in ${}_{H}\mathcal{YD}^{H}$ reads elementwise $\sum b_{[0]}\cdot(b_{[1]}\triangleright a) = a\cdot b$. As a minor lapsus, it is wrongly stated in~\cite{borowpachol} that the monoid condition in ${}_H\mathcal{YD}^H$ implies braided commutativity.

If $H$ is quasitriangular with a universal $R$-element $R = \sum R_{1}\otimes R_{2}\in H\otimes H$, there is a particular source of examples~\cite{CohenWestreich}, namely, for any $H$-module $A$, there is a right $H$-coaction
\begin{equation}\label{eq:Rcoact}
  a\mapsto \sum (R_{2}\triangleright a)\otimes R_{1}
\end{equation}
that makes it into a YD $H$-module and, whenever $A$ is actually a left $H$-module algebra which is braided commutative as a monoid in ${}_H\mathcal{M}$, this $H$-coaction makes it into a braided commutative YD $H$-module algebra (\cite{BrzMilitaru}, Example 4.2). The component  $\sigma_{AA}\colon A\otimes A\to A\otimes A$ of the prebraiding is then given by $m\otimes n\mapsto \sum (R_{2} \triangleright n)\otimes (R_{1}\triangleright m)$.

While in~\cite{borowpachol} quasitriangularity is required in the statements, their proofs additionally use formula~(\ref{eq:Rcoact}) for the $H$-coaction of {\em any} braided commutative YD $H$-module algebra. This is unsatisfactory and misleading as implied by the following simple counterexample provided to us by P.~Saracco and J.~Vercruysse~\cite{sarvercr}.

{\bf Counterexample.} Let $G$ be a finite group and $H = \genfd G$ a group algebra viewed as a (triangular, $R=1\otimes 1$) Hopf $\genfd$-algebra via the coaction $\Delta(h) = h\otimes h$ for $h\in H$. Let $A=\genfd G$ with the $H$-action by conjugation, $h\triangleright g = h g h^{-1}$ and the $H$-coaction $g\mapsto g\otimes g^{-1}$ ($g,h\in G$). Then the YD condition follows by the calculation
$(h\triangleright g)\otimes h g^{-1} = h g h^{-1} \otimes h g^{-1}
= h g h^{-1} \otimes hg^{-1}h^{-1}h =(h\triangleright g) \otimes (h\triangleright g)^{-1}h$. The braided commutativity is directly checked as $g(g^{-1}\triangleright h) = gg^{-1}hg = hg$.

J.~Vercruysse~\cite{vercr} noted that, generalizing this example, any Hopf algebra $H$ with a bijective antipode $S$, considered as a left $H$-module algebra via $h\triangleright g = h_{(1)}g S(h_{(2)})$ is a braided commutative YD $H$-module algebra via the right coaction $g\mapsto g_{(2)}\otimes S^{-1}(g_{(1)})$. Indeed, the proof for the YD condition above readily generalizes as
$$\begin{array}{lcl}
(h_{(1)}\triangleright g_{(2)})\otimes h_{(2)}S^{-1}(g_{(1)})
&=& h_{(3)}g_{(2)}S(h_{(4)})\otimes h_{(5)}S^{-1}(g_{(1)})S^{-1}(h_{(2)})h_{(1)}\\
&=& h_{(3)}g_{(2)}S(h_{(4)})\otimes S^{-1}(h_{(2)} g_{(1)} S (h_{(5)})) h_{(1)}\\
&=& (h_{(2)}g S(h_{(3)}))_{(2)}\otimes S^{-1}((h_{(2)}g S(h_{(3)}))_{(1)}) h_{(1)}\\
&=& (h_{(2)}\triangleright g)_{[0]}\otimes (h_{(2)}\triangleright g)_{[1]}h_{(1)}.
\end{array}$$
The braided commutativity is again easy:
$$\begin{array}{lcl}
  g_{(2)}(S^{-1}(g_{(1)})\triangleright h)&=& g_{(2)} S^{-1}(g_{(1)})_{(1)} h S(S^{-1}(g_{(1)})_{(2)}) \\
  &=& g_{(3)} S^{-1}(g_{(2)}) h S(S^{-1}(g_{(1)})) = h g.
 \end{array}$$

As a consequence, the proof of the Borowiec--Pacho\l\ main result~(\cite{borowpachol}, Theorem 3.1) has to be redone for a general $H$-coaction on $A$, as exhibited below. 

\section{Twisting entire category of
  Yetter--Drinfeld modules for general bialgebra}

To prove the main result of~\cite{borowpachol} in the case of a general braided commutative YD module algebra $A$, we first study the precise form of twisting for YD $H$-modules and braided commutative YD $H$-module algebras. We use basic notions of monoidal categories and functors~\cite{aguiar,Majid,schauen}.

The triangular counterexample with $H = \genfd G$ in Section~\ref{sec:intro} shows that the quasitriangularity of the bialgebra $H$ does not imply any simplification of the form of the $H$-coaction on $A$ used for the calculations,
so we shall not assume quasitriangularity.

\begin{definition} \cite{Drinfeld,Majid}
A Drinfeld twist for a bialgebra $H$ is an invertible element $F = \sum F^1\otimes F^2\in H\otimes H$ such that $F$ and its inverse $F^{-1}=\sum \bar{F}^1\otimes\bar{F}^2$ satisfy any of the four mutually equivalent 2-cocycle conditions
\begin{equation}\label{eq:cocF}
  F^1\otimes F'^1 F^2_{(1)}\otimes F'^2 F^2_{(2)}
  = F'^1 F^1_{(1)}\otimes F'^2 F^1_{(2)}\otimes F^2,
\end{equation}
\begin{equation}\label{eq:cocbarF}
  \bar{F}^1\otimes\bar{F}^2_{(1)} \bar{F}'^1\otimes\bar{F}^2_{(2)}\bar{F}'^2
  =  \bar{F}^1_{(1)}\bar{F}'^1\otimes\bar{F}^1_{(2)}\bar{F}'^2\otimes\bar{F}^2,
\end{equation}
\begin{equation}\label{eq:cocFmix1}
  \bar{F}^1\otimes\bar{F}^2F^1\otimes F^2
  =F^1_{(1)}\bar{F}^1\otimes F^1_{(2)}\bar{F}^2_{(1)}\otimes F^2\bar{F}^2_{(2)},
\end{equation}
\begin{equation}\label{eq:cocFmix2}
  F^1\otimes\bar{F}^1 F^2\otimes\bar{F}^2
  =F^1\bar{F}^1_{(1)}\otimes F^2_{(1)}\bar{F}^1_{(2)}\otimes F^2_{(2)}\bar{F}^2,
\end{equation}
and such that $F$ is counital, that is, $(\epsilon\otimes\id)(F) = (\id\otimes\epsilon)(F) = 1$. 

Note that the inverse $F^{-1}$ automatically satisfies counitality as well.
\end{definition}

For a Drinfeld 2-cocycle $F$, and $H$ a bialgebra, the $F$-twisted bialgebra $H^F$ is $H$ as an algebra, with the coproduct $h\mapsto F\Delta(h)F^{-1}$~(\cite{Majid}), also denoted $h\mapsto \sum h_{(1)F}\otimes h_{(2)F}$. Regarding that $H^F = H$ as algebras, if $M$ is any $H$-module, it can be viewed as an $H^F$-module $M_F$ with the same action; however, due to the new coproduct $\Delta^F$, the monoidal product $\otimes^F$ on the category ${}_{H^F}\mathcal{M}$ of $H^F$-modules is different from the tensor product $\otimes$ of the underlying $H$-modules (that is, in ${}_H\mathcal{M}$), $M_F\otimes^F N_F$ is $M\otimes N$ as a vector space with the $H^F$-action $h\triangleright_{H^F}(m\otimes n) = (F^1 h_{(1)}\bar{F}^1\triangleright m)\otimes(F^2 h_{(2)}\bar{F}^2\triangleright n)$.
\begin{lemma}\label{lem:zetaMN}
(\cite{Drinfeld}) The functor $A\mapsto A_F$ together with the isomorphisms
\begin{equation}\label{eq:zeta}
\zeta_{M,N}\colon M_F\otimes^F N_F\to (M\otimes N)_F,\,\,\,\,\,m\otimes n\mapsto\sum (\bar{F}^1\triangleright m)\otimes(\bar{F}^2\triangleright n)
\end{equation}
forms an equivalence of the monoidal categories ${}_{H}\mathcal{M}$ and ${}_{H^F}\mathcal{M}$.
\end{lemma}
Strong monoidal functors (hence equivalences of monoidal categories in particular)
send (co)monoids to (co)monoids (\cite{aguiar}, 3.4 and~\cite{schauen}) and braided strong monoidal functors preserve bimonoids and (braided) commutative monoids in symmetric (and even prebraided) monoidal categories. We use this freely within this section. 
In particular, if $A$ is a monoid in ${}_{H}\mathcal{M}$ (that is, an $H$-module algebra), then $A_F$ is an $H^F$-module algebra with the multiplication $a\cdot_F b = (\bar{F}_1\triangleright a)\cdot(\bar{F}_2\triangleright b)$. The monoidal center construction is not functorial but it is under monoidal equivalences. Thus, it follows that the center $\mathcal{Z}({}_H\mathcal{M})$ of ${}_{H}\mathcal{M}$ and the center $\mathcal{Z}({}_{H^F}\mathcal{M})$ of ${}_{H^F}\mathcal{M}$ are braided monoidally equivalent and braided commutative monoids in these centers are also in 1-1 correspondence. If $H$ is in fact a {\it finite-dimensional Hopf algebra}, then the category ${}_H\mathcal{YD}^H$ is monoidally equivalent to the center $\mathcal{Z}({}_H\mathcal{M})$, hence there is a braided monoidal equivalence between ${}_H\mathcal{YD}^H$ and ${}_{H^F}\mathcal{YD}^{H^F}$ and, in particular, the categories of braided commutative monoids in them are equivalent, hence an induced correspondence of braided commutative monoids. Chen and Zhang~\cite{chenzhang2007} computed the formulas for the equivalence between ${}_H\mathcal{YD}^H$ and ${}_{H^F}\mathcal{YD}^{H^F}$ for any finite-dimensional Hopf algebra $H$. As a technical result useful to derive a generalization of Borowiec--Pacho\l\ main result on twisting bialgebroids, we now prove that Chen--Zhang formula for twisted coaction (see~(\ref{eq:rhoF}) below) provides a YD structure on $M_F$ for $M$ a YD module over a {\em general bialgebra} $H$, leading to an equivalence of prebraided monoidal categories.
 \begin{theorem}\label{thm:CZgen}
Suppose $H$ is any bialgebra, $F$ a Drinfeld twist for $H$ and $(M,\triangleright,\rho)$ a left-right YD $H$-module. Then the following holds.

 (i) The $\genfd$-linear map $\rho^F\colon M_F\to M_F\otimes H^F$ given by
\begin{equation}\label{eq:rhoF}
  \rho^F(m) = \sum m_{[0]F}\otimes m_{[1]F} := \sum F^1\triangleright (\bar{F}^2\triangleright m)_{[0]}
  \otimes F^2 (\bar{F}^2\triangleright m)_{[1]}\bar{F}^1
 \end{equation}
 is a right $H^F$-coaction and $(M_F,\triangleright,\rho^F)$ is a Yetter--Drinfeld $H^F$-module.

   (ii) If $N$ is also a YD $H$-module, the isomorphisms $\zeta_{M,N}$ are
 (iso)morphisms of $H^F$-comodules, hence also of YD $H^F$-modules.

   (iii) The functor $(M,\triangleright,\rho)\to(M_F,\triangleright,\rho^F)$ (identity on morphisms) together with the isomorphisms $\zeta_{M,N}$ of YD $H^F$-modules forms a monoidal equivalence ${}_H\mathcal{YD}^H\to{}_{H^F}\mathcal{YD}^{H^F}$ lifting the monoidal equivalence ${}_H\mathcal{M}\to{}_{H^F}\mathcal{M}$.

 (iv) If $A$ is a YD $H$-module algebra, then $A_F$ is a YD $H^F$-module algebra.

 (v) If in (iv) $A$ is braided commutative, then $A_F$
 is as well.

 (vi) The equivalence ${}_H\mathcal{YD}^H\to{}_{H^F}\mathcal{YD}^{H^F}$ in (iii) is prebraided monoidal.
 \end{theorem}
\begin{proof}
  (i) We first show that $\rho^F$ is an $H^F$-coaction, $(\rho^F\otimes\id)\rho^F = (\id\otimes\Delta^F)\rho^F$.
$$\begin{array}{l}
   (\rho^F\otimes\id)\rho^F(m) \stackrel{(\ref{eq:rhoF})}=
   (\rho^F\otimes\id)((F^1\triangleright (\bar{F}^2\triangleright m)_{[0]})
  \otimes F^2 (\bar{F}^2\triangleright m)_{[1]}\bar{F}^1)
\\ =
F'^1\triangleright (\bar{F}'^2\triangleright(F^1\triangleright(\bar{F}^2\triangleright m)_{[0]}))_{[0]}\otimes F'^2(\bar{F}'^2\triangleright(F^1\triangleright(\bar{F}^2\triangleright m)_{[0]}))_{[1]}\bar{F}'^1\otimes F^2(\bar{F}^2\triangleright m)_{[1]}\bar{F}^1
\\ = F'^1\triangleright (\bar{F}'^2 F^1\triangleright(\bar{F}^2\triangleright m)_{[0]})_{[0]}\otimes F'^2(\bar{F}'^2 F^1\triangleright (\bar{F}^2\triangleright m)_{[0]})_{[1]}\bar{F}'^1\otimes F^2(\bar{F}^2\triangleright m)_{[1]}\bar{F}^1
  \\
\stackrel{(\ref{eq:cocFmix1})}=
F'^1\triangleright (F^1_{(2)}\bar{F}'^2_{(1)}\triangleright(\bar{F}^2\triangleright m)_{[0]})_{[0]}\otimes F'^2(F^1_{(2)}\bar{F}'^2_{(1)}\triangleright (\bar{F}^2\triangleright m)_{[0]})_{[1]}F^1_{(1)}\bar{F}'^1\otimes F^2\bar{F}'^2_{(2)}(\bar{F}^2\triangleright m)_{[1]}\bar{F}^1
\\ \stackrel{(\ref{eq:YD})}=
F'^1 F^1_{(1)}\triangleright(\bar{F}'^2_{(1)}\triangleright(\bar{F}^2\triangleright m)_{[0]})_{[0]}\otimes F'^2 F^1_{(2)}(\bar{F}'^2_{(1)}\triangleright(\bar{F}^2\triangleright m)_{[0]})_{[1]}\bar{F}'^1\otimes F^2 \bar{F}'^2_{(2)}(\bar{F}^2\triangleright m)_{[1]}\bar{F}^1
\\ \stackrel{(\ref{eq:cocF})}=
F^1\triangleright(\bar{F}'^2_{(1)}\triangleright(\bar{F}^2\triangleright m)_{[0]})_{[0]}\otimes F'^1 F^2_{(1)}(\bar{F}'^2_{(1)}\triangleright(\bar{F}^2\triangleright m)_{[0]})_{[1]}\bar{F}'^1\otimes F'^2 F^2_{(2)}\bar{F}'^2_{(2)}(\bar{F}^2\triangleright m)_{[1]}\bar{F}^1
\\ \stackrel{(\ref{eq:YD})}=
   F^1\triangleright(\bar{F}'^2_{(2)}\bar{F}^2\triangleright m)_{[0]}\otimes F'^1 F^2_{(1)}(\bar{F}'^2_{(2)}\bar{F}^2\triangleright m)_{[1]}\bar{F}'^1\otimes F'^2 F^2_{(2)}(\bar{F}'^2_{(2)}\bar{F}^2\triangleright m)_{[2]}\bar{F}'^2_{(1)} \bar{F}^1
    \\ \stackrel{(\ref{eq:cocbarF})}=
  F^1\triangleright(\bar{F}^2\triangleright m)_{[0]}\otimes F'^1 F^2_{(1)}(\bar{F}^2\triangleright m)_{[1]}\bar{F}^1_{(1)}\bar{F}'^1\otimes F'^2 F^2_{(2)}(\bar{F}^2\triangleright m)_{[2]}\bar{F}^1_{(2)}\bar{F}'^2
    \\ = F^1\triangleright (\bar{F}^2\triangleright m)_{[0]}\otimes F'\Delta(F^2 (\bar{F}^2\triangleright m)_{[1]}\bar{F}^1)F'^{-1}
\\ = (\id\otimes\Delta^F)\rho^F(m).
  \end{array}$$
The YD condition for the $H$-coaction $\rho\colon m\mapsto \sum m_{[0]}\otimes m_{[1]}$ implies the YD condition for the new $H^F$-coaction $\rho^F\colon m\mapsto \sum m_{[0]F}\otimes m_{[1]F}$ by calculation
$$\begin{array}{l}
  (h_{(1)F}\triangleright m_{[0]F})\otimes (h_{(2)F}  m_{[1]F}) =
  \\
  = ((F^1h_{(1)}\bar{F}^1)\triangleright (F'^1\triangleright (\bar{F}'^2\triangleright m)_{[0]}))\otimes F^2 h_{(2)}\bar{F}^2 F'^2(\bar{F}'^2\triangleright m)_{[1]}\bar{F}'^1\\
  = F^1\triangleright
  (h_{(1)}\triangleright (\bar{F}'^2\triangleright m)_{[0]})\otimes
  F^2 (h_{(2)}(\bar{F}'^2\triangleright m)_{[1]})\bar{F}'^1
  \\
  \stackrel{(\ref{eq:YD})}= F^1\triangleright(h_{(2)}\triangleright (\bar{F}'^2\triangleright m))_{[0]}\otimes F^2 (h_{(2)}\triangleright (\bar{F}'^2\triangleright m))_{[1]}h_{(1)}\bar{F}'^1
  \\
= F^1\triangleright(\bar{F}^2 \triangleright(F'^2 h_{(2)}\bar{F}'^2\triangleright m))_{[0]}\otimes F^2(\bar{F}^2\triangleright (F'^2 h_{(2)}\bar{F}'^2\triangleright m))_{[1]}\bar{F}^1 F'^1 h_{(1)}\bar{F}'^1
\\
=(h_{(2)F}\triangleright m)_{[0]F}\otimes (h_{(2)F}\triangleright m)_{[1]F}h_{(1)F}.
\end{array}$$
(ii) We first calculate
$$\begin{array}{l}
\rho_{M_F\otimes^F N_F}(m\otimes n) \stackrel{(\ref{eq:rhotensor})}{=} m_{[0]F} \otimes n_{[0]F} \otimes n_{[1]F}m_{[1]F}
\\
\stackrel{(\ref{eq:rhoF})}{=} F^1\triangleright(\bar{F}'^2\triangleright m)_{[0]}
\otimes F'^1\triangleright(\bar{F}^2\triangleright n)_{[0]}
\otimes F'^2(\bar{F}^2\triangleright n)_{[1]}\bar{F}^1
F^2(\bar{F}'^2\triangleright m)_{[1]}\bar{F}'^1 
\\
\stackrel{(\ref{eq:cocFmix2})}= 
F^1\bar{F}^1_{(1)}\triangleright(\bar{F}'^2\triangleright m)_{[0]}
\otimes F'^1 \triangleright(F^2_{(2)}\bar{F}^2\triangleright n)_{[0]}
\otimes F'^2 (F^2_{(2)}\bar{F}^2\triangleright n)_{[1]}
F^2_{(1)}\bar{F}^1_{(2)}(\bar{F}'^2\triangleright m)_{[1]}\bar{F}'^1 
\\ \stackrel{(\ref{eq:YD})}= 
F^1\bar{F}^1_{(1)}\triangleright(\bar{F}'^2\triangleright m)_{[0]}
\otimes F'^1 F^2_{(1)}\triangleright(\bar{F}^2\triangleright n)_{[0]}
\otimes F'^2 F^2_{(2)}(\bar{F}^2\triangleright n)_{[1]}
\bar{F}^1_{(2)}(\bar{F}'^2\triangleright m)_{[1]}\bar{F}'^1 
\\ \stackrel{(\ref{eq:YD})}=
F^1\triangleright(\bar{F}^1_{(2)}\bar{F}'^2\triangleright m)_{[0]}
\otimes F'^1 F^2_{(1)}\triangleright(\bar{F}^2\triangleright n)_{[0]}
\otimes F'^2 F^2_{(2)}(\bar{F}^2\triangleright n)_{[1]}
(\bar{F}^1_{(2)}\bar{F}'^2\triangleright m)_{[1]}\bar{F}^1_{(1)}\bar{F}'^1 
\\ \stackrel{(\ref{eq:cocbarF})}=
  F^1\triangleright(\bar{F}^2_{(1)}\bar{F}'^1\triangleright m)_{[0]}\otimes
  F'^1 F^2_{(1)}\triangleright(\bar{F}^2_{(2)}\bar{F}'^2\triangleright n)_{[0]}
\otimes F'^2 F^2_{(2)}(\bar{F}^2_{(2)}\bar{F}'^2\triangleright n)_{[1]}
(\bar{F}^2_{(1)}\bar{F}'^1\triangleright m)_{[1]}\bar{F}^1.
\end{array}$$
We observe next that, by the definitions,
for all $m\in M$, $n\in N$,
$$\begin{array}{l}
\rho_{(M\otimes N)_F} (m\otimes n)\stackrel{(\ref{eq:rhoF})}=
F^1\triangleright (\bar{F}^2\triangleright (m\otimes n))_{[0]}\otimes
F^2(\bar{F}^2\triangleright (m\otimes n))_{[1]}\bar{F}^1
\\
\stackrel{(\ref{eq:rhotensor})}=
F^1_{(1)}\triangleright(\bar{F}^2_{(1)}\triangleright m)_{[0]}\otimes
F^1_{(2)}\triangleright(\bar{F}^2_{(2)}\triangleright n)_{[0]}\otimes
F^2(\bar{F}^2_{(2)}\triangleright n)_{[1]}
  (\bar{F}^2_{(1)}\triangleright m)_{[1]}\bar{F}^1,
\end{array}$$
where all action symbols $\triangleright$ are in ${}_H\mathcal{M}$.
It follows that
$$\begin{array}{l}
  ((\zeta_{M,N}^{-1}\otimes\id_H)\circ\rho_{(M\otimes N)_F}\circ\zeta_{M,N})(m\otimes n))
  \\
  = (F'^1\otimes F'^2\otimes 1)\, (\triangleright \otimes \triangleright \otimes \cdot) \, \rho_{(M\otimes N)_F}((\bar{F}'^1\triangleright m)\otimes(\bar{F}'^2\triangleright n))
  \\ =
 F'^1 F^1_{(1)}\triangleright(\bar{F}^2_{(1)}\bar{F}'^1\triangleright m)_{[0]}\otimes
 F'^2 F^1_{(2)}\triangleright(\bar{F}^2_{(2)}\bar{F}'^2\triangleright n)_{[0]}\otimes
  F^2(\bar{F}^2_{(2)}\bar{F}'^2\triangleright n)_{[1]}
  (\bar{F}^2_{(1)}\bar{F}'^1\triangleright m)_{[1]}\bar{F}^1
  \\ \stackrel{(\ref{eq:cocF})}=
F^1\triangleright(\bar{F}^2_{(1)}\bar{F}'^1\triangleright m)_{[0]}\otimes
 F'^1 F^2_{(1)}\triangleright(\bar{F}^2_{(2)}\bar{F}'^2\triangleright n)_{[0]}\otimes
  F'^2F^2_{(2)}(\bar{F}^2_{(2)}\bar{F}'^2\triangleright n)_{[1]}
  (\bar{F}^2_{(1)}\bar{F}'^1\triangleright m)_{[1]}\bar{F}^1.
\end{array}$$
Comparing the results, we obtain
$\rho_{M_F\otimes^F N_F} = (\zeta_{M,N}\otimes\id_H)^{-1}\circ\rho_{(M\otimes N)_F}\circ\zeta_{M,N}$, hence $\zeta_{M,N}$ is indeed a morphism of $H^F$-comodules.

(iii) It is sufficient to observe that the maps forming the functor $(M,\triangleright,\rho)\mapsto (M_F,\triangleright,\rho^F)$ together with the maps
$\zeta_{M,N}\colon M_F\otimes^F N_F\to(M\otimes N)_F$ from~(\ref{eq:zeta}) viewed
by (ii) as morphisms in $\mathcal{M}^{H^F}$ lift the data 
forming the monoidal equivalence ${}_{H}\mathcal{M}^{H}\to{}_{H^F}\mathcal{M}^{H^F}$ from Lemma~\ref{lem:zetaMN} to ${}_{H}\mathcal{YD}^{H}\to{}_{H^F}\mathcal{YD}^{H^F}$ along the forgetful functors
$U\colon{}_H\mathcal{YD}^{H}\to{}_H\mathcal{M}$ and $U^F\colon{}_{H^F}\mathcal{YD}^{H^F}\to{}_{H^F}\mathcal{M}$. Functors $U$ and $U^F$ are faithful strict monoidal, hence all the defining (algebraic) properties to form a monoidal equivalence ${}_{H}\mathcal{YD}^{H}\to{}_{H^F}\mathcal{YD}^{H^F}$ are automatic. 

(iv) A consequence of the monoidal equivalence is that 
for any algebra $A$ in ${}_H\mathcal{YD}^H$,
$A_F$ is automatically an algebra in ${}_{H^F}\mathcal{YD}^{H^F}$.

(v) To show that a braided commutative algebra $A$
is sent to a braided commutative algebra $A_F$, we need to check that if 
$b_{[0]}\cdot(b_{[1]}\triangleright a) = a\cdot b$ for all $a,b\in A$,
we have
$$
b_{[0]F}\cdot_F(b_{[1]F}\triangleright_F a) = a\cdot_F b,\,\,\forall a,b\in A_F.
$$
Now, $\triangleright_F = \triangleright$ and by using~(\ref{eq:rhoF}) for $m = b$,
$$
(F^1\triangleright (\bar{F}^2\triangleright b)_{[0]})\cdot_F ((F^2(\bar{F}^2\triangleright b)_{[1]}\bar{F}^1)\triangleright a) = a\cdot_F b,
$$
which, after rewriting $\cdot_F$ in terms of $\cdot$ and $F$, and after elementary cancellations, gives
$$
(\bar{F}^2\triangleright b)_{[0]}\cdot ((\bar{F}^2\triangleright b)_{[1]}\triangleright (\bar{F}^1\triangleright a)) =  (\bar{F}^1\triangleright a)\cdot (\bar{F}^2\triangleright b),
$$
which indeed holds by the braided commutativity of $A$.

(vi) We need to show 
$\zeta_{N,M}\circ\sigma^F_{M_F,N_F}= \sigma_{M,N}\circ\zeta_{M,N}\colon M_F\otimes^F N_F\to (N\otimes M)_F$
(see~(\ref{eq:sigma}),(\ref{eq:zeta})), where $\sigma^F$ is the prebraiding in ${}_{H^F}\mathcal{YD}^{H^F}$ and $\sigma_{M,N}$ is a component of the prebraiding in ${}_H\mathcal{YD}^H$ understood as a morphism in ${}_{H^F}\mathcal{YD}^{H^F}$. 
This boils down to
\begin{equation}\label{eq:bmoneq}(\bar{F}^1\triangleright n_{[0]F})\otimes (\bar{F}^2 \triangleright (n_{[1]F}\triangleright m))
  = (\bar{F}^2\triangleright n)_{[0]}\otimes((\bar{F}^2\triangleright n)_{[1]}\triangleright(\bar{F}^1\triangleright m)).
\end{equation}
By using formula~(\ref{eq:rhoF}) for $\rho^F(n)$, the left-hand side becomes
$(\bar{F}^1 F^1\triangleright (\bar{F}'^2\triangleright n)_{[0]})\otimes((\bar{F}^2 F^2(\bar{F}'^2\triangleright n)_{[1]}\bar{F}'^1)\triangleright m)$.
After cancellation $\mathcal{F}^{-1}\mathcal{F}= 1\otimes_{A_F} 1$, we easily obtain the equality in~(\ref{eq:bmoneq}).
\end{proof}

\section{Isomorphism of twisted bialgebroids}

Given a bialgebra $H = (H,\mu_H,\eta,\Delta,\epsilon)$ and a left $H$-module algebra $(A,\triangleright)$, the smash product algebra $A\sharp H$ is the tensor product $A\otimes H$ as a vector space, equipped with the associative multiplication bilinearly extending formulas $(a\otimes h)(b\otimes k) = a (h_{(1)}\triangleright b)\otimes h_{(2)} k$, where $a\sharp h$ is the notation for $a\otimes h\in A\otimes H$ in the context of this algebra structure. It is useful to notice the embeddings $s\colon A\to A\sharp H, a\mapsto a\sharp 1$ and $i_H\colon H\to A\sharp H$, $h\mapsto 1\sharp h$, usually viewed as identifications. Viewing $\mu_H$ as an $H$-action on $H$, the monoidal structure on ${}_H\mathcal{M}$ induces an $H$-module structure on $A\otimes H$ by $h\triangleright (a\otimes k) = (h_{(1)}\triangleright a)\otimes h_{(2)}k$. For a twist $F$, observe the $H^F$-module isomorphism $\zeta_{A,H}$ arising as a component of the natural transformation $\zeta$~(\ref{eq:zeta}),
\begin{equation}
  \zeta_{A,H}=(\bar{F}^1\triangleright\otimes\bar{F}^2\cdot)\colon A_F\otimes^F H^F \to (A\otimes H)_F.
\end{equation}
If we equip $A_F\otimes^FH^F$ with the smash product structure $A_F\sharp H^F$ (for the bialgebra $H^F$) and $A\otimes H$ with the smash product structure $A\sharp H$ for the bialgebra $H$, this vector space isomorphism is an algebra isomorphism. This folklore algebra isomorphism has been {\it ad hoc} postulated in~\cite{borowpachol}, Proposition 3.1. Denoting $a\otimes h\in A_F\sharp H^F$ as $a\sharp^F h$, we check that $\zeta_{A,H}$ is an algebra homomorphism,
$$\begin{array}{lcl}
  \zeta_{A,H}(a\sharp^F h)\zeta_{A,H}(b\sharp^F k) &=&
  (\bar{F}^1\triangleright a)\sharp\bar{F}^2 h \cdot (\bar{F}'^1\triangleright b)\sharp\bar{F}'^2 k\\
  &=& (\bar{F}^1\triangleright a)(\bar{F}^2_{(1)}h_{(1)}\bar{F}'^1\triangleright b)
  \sharp\bar{F}^2_{(2)}h_{(2)}\bar{F}'^2 k\\
  &=& (\bar{F}^1\triangleright a)(\bar{F}^2_{(1)}\bar{F}'^1 h_{(1)F}\triangleright b)
  \sharp\bar{F}^2_{(2)}\bar{F}'^2 h_{(2)F} k\\
  &\stackrel{(\ref{eq:cocbarF})}=& (\bar{F}^1_{(1)}\bar{F}'^1\triangleright a)(\bar{F}^1_{(2)}\bar{F}'^2 h_{(1)F}\triangleright b)
  \sharp\bar{F}^2 h_{(2)F} k\\
  &=& \bar{F}^1\triangleright((\bar{F}'^1\triangleright a)\cdot_A (\bar{F}'^2\triangleright (h_{(1)F}\triangleright b)))\sharp \bar{F}^2 h_{(2)F} k\\
  &=& \bar{F}^1\triangleright(a\cdot_{A_F}(h_{(1)F}\triangleright b))\sharp \bar{F}^2 h_{(2)F} k \\
  &=& \zeta_{A,H}((a\cdot_{A_F}(h_{(1)F}\triangleright b))\sharp^F h_{(2)F} k)\\
  &=& \zeta_{A,H}((a\sharp^F h)(b\sharp^F k)).
\end{array}$$
The compatibility with the unit element is direct by the counitality of $F^{-1}$, namely $\zeta_{A,H}(1_A\sharp^F 1_H) = (\bar{F}^1\triangleright 1_A)\otimes\bar{F}^2 = \epsilon(\bar{F}^1)1_A\otimes\bar{F}^2 = 1_A\sharp 1_H$.

For an algebra $A$ (``base algebra''), a {\it left $A$-bialgebroid} is given by data $(\mathcal{H},s,t,\Delta^{\mathcal{H}},\epsilon^{\mathcal{H}})$, where $\mathcal{H}$ is an algebra (``total algebra''), $s\colon A\to\mathcal{H}$ (``source map''), $t\colon A^{\mathrm{op}}\to\mathcal{H}$ (``target map'') are algebra maps with commuting images so that $\mathcal{H}$ is an $A$-bimodule via $a.h.a' = s(a)t(a')h$, $\Delta^{\mathcal{H}}\colon \mathcal{H}\to\mathcal{H}\otimes_A\mathcal{H}$ is a coassociative comultiplication in the category of $A$-bimodules with a counit $\epsilon^{\mathcal{H}}\colon \mathcal{H}\to A$ and several standard axioms are required~\cite{lu,BrzMilitaru,bohmHbk,stojicscext}. A useful datum is a map $\blacktriangleright\colon \mathcal{H}\otimes A\to A$ defined by $h\otimes a\mapsto h\blacktriangleright a = \epsilon^{\mathcal{H}}(h\cdot s(a))$. A {\it bialgebroid 2-cocycle}  is an element $\mathcal{G}\in\mathcal{H}\otimes_A\mathcal{H}$ satisfying bialgebroid versions of the 2-cocycle and counitality conditions for $F^{-1}$~\cite{xu,twlinPois,twosha}. Given a bialgebroid 2-cocycle $\mathcal{G}=\sum\mathcal{G}^1\otimes\mathcal{G}^2$, one defines the twisted base algebra $A^{\mathcal{G}}$ with the same underlying vector space $A$ and the associative multiplication $a\ast b = (\mathcal{G}^1\blacktriangleright a)\cdot_A(\mathcal{G}^2\blacktriangleright b)$, twisted source 
 $s^{\mathcal{G}}(a) = s(\mathcal{G}^1\blacktriangleright a)\mathcal{G}^2$ and target $t^{\mathcal{G}}(a) = t(\mathcal{G}^2\blacktriangleright a)\mathcal{G}^1$.
The maps $s^{\mathcal{G}}\colon A^{\mathcal{G}}\to\mathcal{H}$ and $t^{\mathcal{G}}\colon (A^{\mathcal{G}})^{\mathrm{op}}\to\mathcal{H}$ define a new $A^{\mathcal{G}}$-bimodule structure $\mathcal{H}^{\mathcal{G}}$ on $\mathcal{H}$ and the tensor product bimodule $\mathcal{H}^{\mathcal{G}}\otimes_{A^{\mathcal{G}}}\mathcal{H}^{\mathcal{G}}$. When the confusion does not arise, we write $\mathcal{H}$ for $\mathcal{H}^{\mathcal{G}}$. An {\it invertible} or {\it Drinfeld--Xu bialgebroid 2-cocycle} (a twistor in the terminology of Xu~\cite{xu}) is a bialgebroid 2-cocycle $\mathcal{G}$ having an inverse $\mathcal{G}^{-1}\in\mathcal{H}^{\mathcal{G}}\otimes_{A^{\mathcal{G}}}\mathcal{H}^{\mathcal{G}}$ in the sense $\mathcal{G}^{-1}\mathcal{G}=1\otimes_{A^{\mathcal{G}}} 1$ and $\mathcal{G}\mathcal{G}^{-1}=1\otimes_A 1$.  Given a Drinfeld--Xu 2-cocycle $\mathcal{G}$, a twisted bialgebroid $\mathcal{H}^{\mathcal{G}}=(\mathcal{H}^{\mathcal{G}},s^{\mathcal{G}},t^{\mathcal{G}},\Delta^{\mathcal{H}^{\mathcal{G}}},\epsilon^{\mathcal{H}^{\mathcal{G}}})$ is defined by $\Delta^{\mathcal{H}^{\mathcal{G}}}(x) = \mathcal{G}^{-1\sharp}\Delta^{\mathcal{H}}(x)\mathcal{G}$ and $\epsilon^{\mathcal{H}^{\mathcal{G}}} = \epsilon^{\mathcal{H}}$, where $ \mathcal{G}^{-1\sharp}$ is the multiplication with $\mathcal{G}^{-1}$ from the left as a well defined map $\mathcal{H}\otimes_A\mathcal{H}\to\mathcal{H}^{\mathcal{G}}\otimes_{A^{\mathcal{G}}}\mathcal{H}^{\mathcal{G}}$~\cite{xu}. In~\cite{borowpachol}, the inverse $\mathcal{F}=\mathcal{G}^{-1}$, which is automatically a 2-cocycle for $\mathcal{H}^{\mathcal{G}}$, is considered as the basic 2-cocycle, but we start from Xu's convention because $\mathcal{F}\in\mathcal{H}^{\mathcal{G}}\otimes_{A^{\mathcal{G}}}\mathcal{H}^{\mathcal{G}}$ is well defined in the tensor product which itself needs data of $\mathcal{F}$ (or, equivalently, $\mathcal{G}$) to be properly defined.

If an $H$-module algebra $A$ carries an $H$-coaction so that it becomes a left-right braided commutative YD $H$-module algebra, the smash product algebra $\mathcal{H} = A\sharp H$ has a structure of a left associative $A$-bialgebroid~\cite{BrzMilitaru,bohmHbk,twosha}; it is a Hopf algebroid~\cite{bohmHbk,stojicscext} if $H$ is a Hopf algebra with a bijective antipode. It is an analogue of a transformation or action groupoid, and it is a variant of a construction from~\cite{lu} where, instead of YD module algebras, module algebras over a Drinfeld double were used. Bialgebroid $A\sharp H$ is sometimes called the scalar extension bialgebroid~\cite{bohmHbk} because $\Delta^{\mathcal{H}}$ extends $\Delta \colon  h\mapsto h_{(1)}\otimes h_{(2)}$ along the embedding of $H$ as $H\cong\genfd\otimes H\subset A\sharp H$, namely $\Delta^{\mathcal{H}}(a\sharp h) = a\sharp h_{(1)}\otimes_A 1\sharp h_{(2)}$ with $\epsilon^{\mathcal{H}}\colon  a\sharp h\mapsto a\epsilon(h)$. The $A$-bimodule structure is given by $s\colon a\mapsto a\sharp 1$, $t\colon a\mapsto a_{[0]}\sharp a_{[1]}$; therefore $(a\sharp h)\blacktriangleright b = \epsilon^{\mathcal{H}}(a (h_{(1)}\triangleright b)\sharp h_{(2)}) = a (h\triangleright b)$. 

Drinfeld 2-cocycle $F$ of a bialgebra $H$ induces a Drinfeld--Xu bialgebroid 2-cocycle $\mathcal{G}=(1\sharp \bar{F}^1)\otimes_{A}(1\sharp\bar{F}^2)\in\mathcal{H}\otimes_A\mathcal{H}$ or $\mathcal{F}=\mathcal{G}^{-1}=(1\sharp^F F^1)\otimes_{A_F}(1\sharp^F F^2)$. Notice that the multiplications $\ast$ on $A^{\mathcal{G}}$ and $\cdot_{A_F}$ on $A_F$ coincide. We now show a proper generalization of Borowiec and Pacho\l~\cite{borowpachol}, Theorem 3.1, allowing for general $H$-coaction on $A$ and also beyond quasitriangular case.
\begin{theorem}\label{thm:twbialgd}
For any bialgebra $H$, a Drinfeld twist $F\in H\otimes H$ and a braided commutative YD $H$-module algebra $A$, the component $\zeta_{A,H}\colon A_F\otimes H^F\to(A\otimes H)^{\mathcal{G}}$ of the natural isomorphism $\zeta$ of functors ${}_H\mathcal{YD}^H\to{}_{H^F}\mathcal{YD}^{H^F}$ considered as an algebra isomorphism $\zeta_{A,H}\colon A_F\sharp H^F\to(A\sharp H)^{\mathcal{G}}$ of smash products is an isomorphism of the scalar extension $A_F$-bialgebroid $(A_F\sharp H^F,s^F,t^F,\Delta^{F},\epsilon^{F})$ and the $\mathcal{G}$-twist $((A\sharp H)^{\mathcal{G}},s^{\mathcal{G}},t^{\mathcal{G}},\Delta^{\mathcal{H}^{\mathcal{G}}},\epsilon^{\mathcal{H}^{\mathcal{G}}})$ of the scalar extension bialgebroid $(A\sharp H,s,t,\Delta^{\mathcal{H}},\epsilon^{\mathcal{H}})$.
\end{theorem}
\begin{proof}
This means that $\zeta_{A,H}$ commutes with the bialgebroid structure maps appropriately~\cite{bohmHbk}, which is checked as follows:
$$\begin{array}{lcl}
  (\zeta_{A,H}\otimes\zeta_{A,H})(\Delta^F(a\sharp^F h)) 
  &=& (\zeta_{A,H}\otimes\zeta_{A,H})(a\sharp^F h_{(1)F}\otimes 1\sharp^F h_{(2)F})
  \\&=&
  (\bar{F}^1\triangleright a)\sharp\bar{F}^2h_{(1)F}\otimes 1\sharp h_{(2)F}
  \\ &=&
(\bar{F}^1\triangleright a)\sharp\bar{F}^2 F^1h_{(1)}\bar{F}'^1\otimes 1
  \sharp F^2 h_{(2)}\bar{F}'^2
  \\ &\stackrel{(\ref{eq:cocFmix1})}=&
(F^1_{(1)}\bar{F}^1\triangleright a)\sharp F^1_{(2)}\bar{F}^2_{(1)}h_{(1)}\bar{F}'^1
  \otimes 1\sharp F^2\bar{F}^2_{(2)}h_{(2)}\bar{F}'^2
  \\ &=& (1\sharp F^1\otimes 1\sharp F^2)((\bar{F}^1\triangleright a)\sharp\bar{F}^2_{(1)}h_{(1)}\otimes 1\sharp\bar{F}^2_{(2)}h_{(2)})
  (1\sharp\bar{F}'^1\otimes1\sharp\bar{F}'^2)
  \\ &=&
\Delta^{\mathcal{H}^{\mathcal{G}}}((\bar{F}^1\triangleright a)\sharp\bar{F}^2 h)
\\ &=& (\Delta^{\mathcal{H}^{\mathcal{G}}}\circ\zeta_{A,H})(a\sharp^F h),
\end{array}$$
 $$
\zeta_{A,H}(s^F(a)) = \zeta_{A,H}(a\sharp^F 1) = (\bar{F}^1\triangleright a)\sharp\bar{F}^2 =
s(\mathcal{G}^1\blacktriangleright a)\mathcal{G}^2 = s^{\mathcal{G}}(a).
$$
Finally, the check for the target map
is essentially more general than in~\cite{borowpachol}:
$$\begin{array}{lcl}
  \zeta_{A,H}(t^F(a)) &\stackrel{(\ref{eq:rhoF})}=& \bar{F}'^1 F^1\triangleright (\bar{F}^2\triangleright a)_{[0]} \sharp \bar{F}'^2 F^2 (\bar{F}^2\triangleright a)_{[1]}\bar{F}^1
  \\& =&
(\bar{F}^2\triangleright a)_{[0]}\sharp (\bar{F}^2\triangleright a)_{[1]}\bar{F}^1  \\ &= &
  t(\bar{F}^2\triangleright a)(1\sharp\bar{F}^1)
  \\ &=& t^{\mathcal{G}}(a).
\end{array}$$
\end{proof}

\section{Concluding remarks}

The $A$-bialgebroid structure on the smash product algebras $A\sharp H$ depends on the coaction of the Yetter--Drinfeld $H$-module $A$. The main purpose of this article was to correct the consequences of the statement in~\cite{borowpachol}
that for a quasitriangular bialgebra $H$ this coaction of $A$ must be of the special form, hence the target map of $A\sharp H$ (and the right $H$-module structure on $A\sharp H$) would be of special form as well. This is shown false by the counterexample in Section~\ref{sec:intro}, but that such Yetter--Drinfeld module algebras form just a subcategory was known to experts before (see~\cite{chenzhang2007}).

We have not exhibited fully fledged physical examples which are not within the Borowiec--Pacho\l\ framework in~\cite{borowpachol}. New examples of bialgebroids are usually quite involved, so this task is left for the future, but it is very likely that many such examples are in place having in mind that the basic counterexample is very simple and that smash products $A\sharp H$ are ubiquitious in noncommutative geometry (including nonquasitriangular case). Our past works~\cite{halgoid,twlinPois,lukierskicov} and Borowiec--Pacho\l\ works are mainly focused on smash products $A\sharp H$ which may be interpreted as noncommutative phase spaces, possibly with included additional symmetries (covariant phase spaces). There are other physical sources of scalar extension bialgebroids. For example, Semikhatov studied (truncations of) Heisenberg doubles $H^*\sharp H$ of some finite-dimensional quantum groups in connection to Kazhdan--Lusztig duality between logarithmic conformal field theories and quantum groups~\cite{semikhatov}. Also, Yetter--Drinfeld module algebras appear in topological field theories. Thus, we think that a solid general treatment of twists of bialgebroids of the type $A\sharp H$ beyond Borowiec--Pacho\l\ case is physically sound.

Our approach leading to Theorem~\ref{thm:twbialgd} is different from~\cite{borowpachol} in the sense that we study twisting of monoidal categories ${}_H\mathcal{M}^H$ and ${}_H\mathcal{YD}^H$ as the main tool, including proving Theorem~\ref{thm:CZgen} (of independent interest), and exhibiting the isomorphism $\zeta_{A,H}$ (denoted $\phi$ in~\cite{borowpachol}) as a particular component of the natural equivalence $\zeta$ of monoidal functors ${}_H\mathcal{YD}^H\longrightarrow{}_{H^F}\mathcal{YD}^{H^F}$. Nevertheless, a more elementary and fully conceptual understanding why this natural equivalence preserves precisely the scalar extension bialgebroid structure is missing.

\section*{Acknowledgements}

A.~Borowiec kindly shared his ideas with the mathematical physics group in Zagreb when visiting us in the early phase of his project leading to~\cite{borowpachol}; the first author (Z.\,\v{S}.) thanks also for his later kind correspondence. However, while Z.\,\v{S}.\ did suggest to avoid assuming quasitriangularity and using in deriving the main result the twist of the entire category of YD modules similarly to the logic shown above, Z.\,\v{S}.\ failed to realize that the actual proof in~\cite{borowpachol} does not cover general scalar extension bialgebroids even for quasitriangular bialgebras and noticed this only as late as early 2022. We thank P.~Saracco and J.~Vercruysse for kind communication of counterexamples~\cite{sarvercr,vercr} and T.~Brzezi\'nski for initiating that communication.

\section*{References}

\bibliographystyle{amsalpha}

\end{document}